\documentclass[final]{dmtcs-episciences}
\usepackage[utf8]{inputenc}

\author{Evan Chen\affiliationmark{1}
	\and Shyam Narayanan\affiliationmark{2}}
\title[Length five quasi-consecutive patterns]%
	{The 26 Wilf-equivalence classes of length
	five quasi-consecutive patterns\thanks{This research
	was funded by NSF grant 1358659 and NSA grant H98230-16-1-0026
	as part of the 2016 Duluth Research Experience for Undergraduates (REU).}}

\affiliation{
  Massachusetts Institute of Technology, Cambridge, MA \\
  Harvard University, Cambridge, MA}
\received{2017-10-30}
\revised{2018-7-17}
\accepted{2018-10-9}

\usepackage{amsmath,amsthm,amssymb}
\usepackage[utf8]{inputenc}
\usepackage{hyperref}
\usepackage{todonotes}
\usepackage[shortlabels]{enumitem}
\usepackage{mathtools}
\usepackage{microtype}
\usepackage{ytableau}
\usepackage{amssymb}

\newtheorem{theorem}{Theorem}[section]
\newtheorem*{theorem*}{Theorem}
\newtheorem{lemma}[theorem]{Lemma}
\newtheorem*{lemma*}{Lemma}
\newtheorem{proposition}[theorem]{Proposition}
\newtheorem*{proposition*}{Proposition}

\theoremstyle{definition}
\newtheorem{definition}[theorem]{Definition}
\newtheorem*{definition*}{Definition}
\newtheorem{example}[theorem]{Example}
\newtheorem*{example*}{Example}
\newtheorem{defn}[theorem]{Definition}
\newtheorem*{defn*}{Definition}
\newtheorem{corollary}[theorem]{Corollary}
\newtheorem*{corollary*}{Corollary}

\renewcommand{\-}{\text{-}}

\newcommand{\ii}{\item}

\newcommand{\fs}{\equiv_{\textrm{fs}}}

\begin{document}
\publicationdetails{20}{2018}{2}{12}{4030}
\maketitle
\begin{abstract}
	We present two families of Wilf-equivalences
	for consecutive and quasi-consecutive vincular patterns.
	These give new proofs of the classification of consecutive
	patterns of length $4$ and $5$.
	We then prove additional equivalences to explicitly classify all
	quasi-consecutive patterns of length $5$ into 26 Wilf-equivalence classes.
\end{abstract}

\section{Introduction}
In the theory of pattern avoidance,
vincular patterns (generalizing classical patterns)
were introduced by Babson and Steingr\'{\i}msson in \cite{Babson2000}.
The aim of this paper is to prove several new Wilf-equivalences
for certain such patterns of length five.

First, we recall the definition of a classical pattern:
\begin{definition}
	We say two sequences $(a_1, \dots, a_k)$ and $(b_1, \dots, b_k)$
	of positive integers are \emph{order-isomorphic} if
	for any $1 \le i,j \le k$, we have $a_i < a_j$ if and only if $b_i < b_j$.
	In this case, we write $a_1 \cdots a_k \sim b_1 \cdots b_k$.
\end{definition}
\begin{definition}
	Let $\sigma \in S_k$
	(here $S_k$ is the set of permutations of $\{1, \dots, k\}$).
	A permutation $\pi \in S_n$ ($n \ge k$) \emph{contains} $\sigma$ as a \emph{classical pattern}
	if there exists $1 \le i_1 < \dots < i_k \le n$ such that 
	\[ \pi_{i_1} \pi_{i_2} \cdots \pi_{i_k} \sim \sigma_1\sigma_2\cdots\sigma_k. \]
	Alternatively, if a permutation does not contain $\sigma$,
	it is said to \emph{avoid} $\sigma$.
\end{definition}
\begin{example}
	The permutation $146235$ avoids the classical pattern $321$.
\end{example}

We now define a so-called vincular pattern,
which is a classical pattern with the additional requirement
that certain pairs of indices must be consecutive.
\begin{definition}
	A \emph{vincular pattern} is a combination of a
	permutation and a set of adjacencies,
	which is formally a pair $(\sigma, T)$
	where $\sigma \in S_k$ and $T \subset \{1, 2, \dots, k-1\}$
	(where $i \in T$ signifies that $i$ and $i+1$ should be adjacent).
	
    We generally denote this by a dashed pattern,
	in which $\pi_i$ and $\pi_{i+1}$ are joined by a dash if $i \notin T$.
	For example, $(2413, \{1,3\})$ is usually written $24\-13$.
\end{definition}

\begin{definition}
	We say a permutation $\pi \in S_n$ \emph{contains} the vincular pattern
	$\sigma = (\sigma_1\cdots\sigma_k, T)$ if there exists $1 \le i_1 < \dots < i_k \le n$ such that 
	\[ \pi_{i_1} \pi_{i_2} \cdots \pi_{i_k} \sim \sigma_1\sigma_2\cdots\sigma_k
		\qquad \text{and}\quad i_{j+1} = i_j+1 \;
		\text{for all } j \in T  \]
	and otherwise $\pi$ \emph{avoids} $\sigma$.

	We then define the following special types of vincular patterns:
	\begin{itemize}
		\ii A \emph{classical} pattern (as defined above) is then a pattern
		of the form $\sigma_1 \- \sigma_2 \- \cdots \- \sigma_n$ (all possible dashes).
		Thus vincular patterns subsume classical patterns.
		\ii We say a pattern of the form $\sigma_1 \sigma_2 \cdots \sigma_n$
		(with no dashes) is a \emph{consecutive} pattern.
		\ii We say a pattern of the form $\sigma_1 \sigma_2 \cdots \sigma_{n-1}\-\sigma_n$
		(with exactly one dash at the end) is a \emph{quasi-consecutive} pattern.
	\end{itemize}
\end{definition}
\begin{example}
	The permutation $\pi = 3275164$ contains $24\-13$ (via $\pi_2\pi_3\pi_5\pi_6$).
\end{example}

The following is the key definition in the study of pattern avoidance.
\begin{definition}
	For each $n$ and a vincular pattern $\sigma$,
	we define $S_n(\sigma)$ as the set of permutations in $S_n$ that avoid $\sigma$.
	We say $\sigma$ and $\tau$ are \emph{Wilf-equivalent}
	(written as $\sigma \equiv \tau$) if 
	\[ \left\lvert S_n(\sigma) \right\rvert
		= \left\lvert S_n(\tau) \right\rvert \]
	for every positive integer $n$.
\end{definition}

\begin{example}
	Trivial examples of Wilf-equivalences include the
	\emph{reverse} and \emph{complement} of vincular patterns.  

	For example, if $\sigma = 12\-4\-3$,
	then the reverse $\sigma^r = 3\-4\-21$ is Wilf-equivalent to $\sigma$
	because there is an obvious bijection $S_n(\sigma) \to S_n(\sigma^r)$
	by reversing the letters.
	Similarly, the complement $\sigma^c = 43\-1\-2$ is also
	Wilf-equivalent to $\sigma$ since if $\pi = \pi_1\pi_2\ldots\pi_n$ avoids $\sigma$,
	then its complement $\pi^c = (n+1-\pi_1)\cdots(n+1-\pi_n)$ avoids $\sigma^c$,
	and vice versa.
	Combining both gives us the equivalences
	$12\-4\-3 \equiv 43\-1\-2 \equiv 2\-1\-34 \equiv 3\-4\-21.$
\end{example}

Wilf-equivalence is clearly an equivalence relation, and thus groups vincular patterns of a fixed length into equivalence classes.
The core question in the theory of Wilf equivalences is to
classify patterns completely into such equivalence classes.

We comment briefly on what is already known;
more details are provided in Section~\ref{sec:table}.
The Wilf-equivalence classification for vincular patterns of length $3$
was completed by Claesson \cite{Claesson2001}, and while
the classification for vincular patterns of length $4$ has not yet been completed,
results by Baxter and Shattuck \cite{BaxterShattuckMainPaper}
and earlier results from
\cite{Baxter2013,Elizalde2006,ElizaldeNoy2003,Kasraoui2013,Kitaev2005}
reduced the problem to two remaining specific conjectures.
However, results from \cite{ElizaldeNoy2003,Nakamura2011,ElizaldeNoy2012}
have fully completed the consecutive case for $k \in \{4,5,6\}$.
Wilf-equivalences of vincular patterns have also been studied in multisets instead of in permutations, such as in Mansour and Shattuck \cite{MansourShattuck2015}.
Wilf-equivalence classification for classical patterns
up to $k = 7$ has been completed,
with both the $k = 6$ and $k = 7$ case completed by Stankova and West
\cite{StankovaWest2002}.
Wilf-equivalence classification for involutions avoiding classical patterns up to $k = 7$ has also been completed, where two patterns are considered equivalent if for all $n$ the number of involutions in $S_n$ avoiding the patterns are equal.  The cases of $k = 5, 6, 7$ were completed by Dukes et al. \cite{DukesEtAl2009}.

In this paper, we focus on \emph{quasi-consecutive} vincular patterns,
which are patterns of the form $\sigma_1\sigma_2\cdots\sigma_{k} \- \sigma_{k+1}$. 
The structure of this paper is as follows.
	In Section~\ref{sec:table}
	we list the table of equivalence classes and describe
	exactly which equivalences remain to be shown.
	In Section~\ref{sec:BijectionCons},
	we provide a bijective result (Theorem~\ref{thm:BijectionCons}),
	which gives proofs for Wilf-equivalences for certain
    consecutive patterns and quasi-consecutive vincular patterns.  
    This leads to a new proof for all Wilf-equivalences 
    for consecutive patterns of lengths $4$ and $5$.
	In Section~\ref{sec:sandcastle},
	we give a second family of equivalences (Theorem~\ref{thm:sandcastle}).
	Finally, in Section~\ref{sec:M} we prove one last equivalence $2153\-4 \equiv 3154\-2$.
Taking all this together we have the following theorem.
\begin{theorem}
	There are exactly 26 equivalence classes of
	quasi-consecutive patterns of length $5$,
	as listed in Table~\ref{tab:classes}.
\end{theorem}

\section{The 26 equivalence classes of quasi-consecutive patterns}
\newcommand{\EK}{\cite{Elizalde2006,Kitaev2005}}
\label{sec:table}

\subsection{Table of equivalence classes}
The claimed $26$ equivalence classes are presented in Table~\ref{tab:classes}.

\begin{table}[ht]
	\centering
	\small
	\begin{tabular}{|crll|}
		\hline
		& \# & Elements & Proof \\ \hline
		A & 2 & 1254-3, 1354-2 & Thm.~\ref{thm:sandcastle} \\
		B & 2 & 1453-2, 1543-2 & Thm.~\ref{thm:BijectionVinc} \\
		C & 2 & 2135-4, 2145-3 & Thm.~\ref{thm:sandcastle} \\
		D & 1 & 1534-2 & Singleton \\
		E & 1 & 2315-4 & Singleton \\
		F & 6 & 1243-5, 1253-4, 2134-5, 2354-1, 3145-2, 3245-1 & See \S\ref{sec:E1} \\
		G & 2 & 3125-4, 3215-4 & Thm.~\ref{thm:BijectionVinc} \\
		H & 1 & 1523-4 & Singleton \\
		I & 1 & 3251-4 & Singleton \\
		J & 1 & 2154-3 & Singleton \\
		K & 1 & 1435-2 & Singleton \\
		L & 16 & 3124-5, 3214-5, 2543-1, 2453-1, 2341-5, 2431-5, & See \S\ref{sec:G} \\
		&& 1342-5, 1432-5, 1352-4, 1452-3, 1532-4, 1542-3, & \\
		&& 2531-4, 2541-3, 2351-4, 2451-3, & \\
		M & 2 & 2153-4, 3154-2 & Thm.~\ref{thm:H} \\
		N & 1 & 2513-4 & Singleton \\
		O & 1 & 1524-3 & Singleton \\
		P & 1 & 2415-3 & Singleton \\
		Q & 1 & 1325-4 & Singleton \\
		R & 4 & 1423-5, 2314-5, 2534-1, 3241-5 & \EK \\
		S & 2 & 2143-5, 3254-1 & \EK \\
		T & 1 & 3152-4 & Singleton \\
		U & 1 & 1425-3 & Singleton \\
		V & 2 & 1324-5, 2435-1 & \EK \\
		W & 1 & 2514-3 & Singleton \\
		X & 2 & 2413-5, 3142-5 &  \EK \\
		Y & 3 & 1235-4, 1245-3, 1345-2 & \cite[Thm.~9]{BaxterShattuckMainPaper} \\
		Z & 2 & 1234-5, 2345-1 & \EK \\ \hline
	\end{tabular}
	\medskip
	\caption{The 26 $abcd$-$e$ equivalence classes.}
	\label{tab:classes}
\end{table}

A few remarks are in order.
For a pattern and its complement,
we only record the lexicographically earliest one
(thus halving the number of entries).
The naming of the classes above
is in ascending order by the number of avoiding permutations when $n=8$.
Distinct equivalence classes which are tied
are then sorted according to the number of permutations when $n=9$, then $n=10$.
(See Appendix~\ref{sec:oeis} for these totals.)
Within each row except L, the permutations are sorted in lexicographically ascending order.

\subsection{Equivalences already known}
The following theorem was established both by
Elizalde \cite{Elizalde2006} and Kitaev \cite{Kitaev2005}.
\begin{theorem}
	[\cite{Elizalde2006,Kitaev2005}]
	\label{thm:EK}
	Suppose that $\sigma_1\sigma_2...\sigma_k \equiv \tau_1\tau_2...\tau_k$.
	Then: \[ \sigma_1\cdots\sigma_k \- (k+1) \equiv \tau_1\tau_2\cdots\tau_k \- (k+1). \]
\end{theorem}
This theorem leads to several corollaries for
quasi-consective patterns of length $5$.

The following result produces
equivalence class Y in Table~\ref{tab:classes}.
\begin{theorem}
	[Theorem 9 of \cite{BaxterShattuckMainPaper}]
	For a fixed $k$, and for $1 \le i \le n-1,$
	define \[ \sigma_i = 12\cdots i(i+2)\cdots(k+1)\-(i+1). \]
	Then, $\sigma_i \equiv \sigma_j$ for every $1 \le i < j \le n-1.$
\end{theorem}

\subsection{Equivalence class F}
\label{sec:E1}
Here we elaborate on equivalence class F.
First, by using \EK\ alongside with the complement operation, we obtain
\[ 1243\-5 \equiv 2134\-5 \equiv 2354\-1 \equiv 3245\-1. \]
Thus, the class F is complete once we prove
\begin{align*}
	1243\-5 & \equiv 1253\-4 \tag{Fa} \\
	3245\-1 &\equiv 3145\-2 \tag{Fb}
\end{align*}
which both follow as corollaries of Theorem~\ref{thm:sandcastle}.

\subsection{Equivalence class L}
\label{sec:G}
We now elaborate on equivalence class L.
Again, \EK\ together with complement operation implies that
the first eight patterns (those ending in $1$ or $5$)
in L are equivalent to one another, that is:
\begin{align*}
	3124\-5 &\equiv 3214\-5\equiv 2543\-1\equiv 2453\-1 \\
	\equiv 2341\-5 &\equiv 2431\-5\equiv 1342\-5\equiv 1432\-5.
\end{align*}

(Here we are using the fact that $1342 \equiv 1432$,
which is for example \cite{ElizaldeNoy2003},
or our Corollary~\ref{BaxterShapeConjecture}.)

Next, by Theorem~\ref{thm:sandcastle} we obtain that
\begin{align*}
	1342\-5 \equiv 1352\-4 &\equiv 1452\-3 \tag{La} \\
	1432\-5 \equiv 1532\-4 &\equiv 1542\-3 \tag{Lb} \\
	2431\-5 \equiv 2531\-4 &\equiv 2541\-3 \tag{Lc} \\
	2341\-5 \equiv 2351\-4 &\equiv 2451\-3. \tag{Ld}
\end{align*}
These finish the classification of L.

\section{A family of bijective filling-shape-Wilf-equivalences}
\subsection{Shape Wilf-Equivalence}
\label{sec:shapewilf}
In order to state the results of Section~\ref{sec:BijectionCons}
in their full generality,
we first define the stronger notions of shape-Wilf-equivalence
and filling-shape-Wilf-equivalence
(introduced in \cite{BabsonWest2000} and \cite{Baxter2013}, respectively).

Consider a Young diagram with $m$ rows and $n$ columns
where columns are nonincreasing from left to right
and rows are nonincreasing from bottom to top.
In other words, the Young diagram can be thought of
as a subset of the $m \times n$ grid where $(1, 1)$ is at the
bottom-left corner and if $(a, b)$ is in the diagram,
then so is $(a', b')$ for any $a' \le a$, $b' \le b$
(this is the ``French notation''). 
We define a filling of a Young Diagram as follows:

\begin{definition}
	A \emph{filling} is a collection of \emph{selected elements} of
	the form $(a_1, b_1)$, \dots, $(a_k, b_k)$ such that
	$a_i \neq a_j$ and $b_i \neq b_j$ for all $1 \le i \neq j \le k$.
    
    Similarly, a \emph{standard filling} is a filling of a
	Young diagram such that $k = m = n,$ i.e.
	every row and every column has an element.
\end{definition}

Examples of nonstandard and standard fillings
are presented in Figure~\ref{onlyfigureprobably}.
    
\begin{figure}[ht]
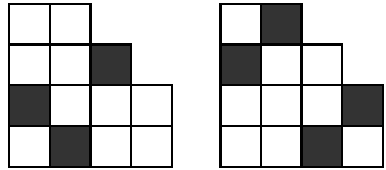

\centering
\begin{ytableau}
\none & &\\
\none & & & *(black!80)\\
\none & *(black!80) & & & \\
\none & & *(black!80) & & \\
\end{ytableau}
\begin{ytableau}
\none & & *(black!80)\\
\none & *(black!80) & & \\
\none & & & & *(black!80) \\
\none & & & *(black!80) & \\
\end{ytableau}
\caption{For the same Young diagram,
a nonstandard filling on the left
and a standard filling on the right.}
\label{onlyfigureprobably}
\end{figure}

We can define containment of a vincular pattern for a filling in the same way as for a permutation.

\begin{definition}
	We say that a filling \emph{contains} $\sigma = (\sigma_1\cdots\sigma_k, T)$ if there exists $i_1, \ldots, i_k$ such that the following conditions hold:
    
\begin{itemize}
	\item $(i_1, \pi(i_1)), \ldots, (i_k, \pi(i_k))$ are selected elements
    \item $\sigma_1\cdots\sigma_k \sim \pi(i_1)\cdots\pi(i_k),$
    \item $(i_j, \pi(i_{j'}))$ is in the Young diagram
		for any $1 \le j, j' \le k$, and
    \item $i_{j+1}-i_j = 1$ for any $j \in T$.
\end{itemize}
    
    Else, we say that a filling \emph{avoids} $\sigma$.
\end{definition}

From here, we can define shape-Wilf-equivalence and filling-shape-Wilf-equivalence.

\begin{definition}
	For two vincular patterns $\sigma, \tau,$ we say that $\sigma$ and $\tau$ are shape-Wilf-equivalent if for any Young diagram, the number of standard fillings avoiding $\sigma$ equals the number of standard fillings avoiding $\tau$. In this case, we write $\sigma \equiv_{\mathrm s} \tau$.
\end{definition}

\begin{definition}
	For two vincular patterns $\sigma, \tau$, we say that $\sigma$ and $\tau$ are filling-shape-Wilf-equivalent if for any Young diagram and for any fixed set of rows $R$ and fixed set of columns $C$, the number of fillings having no element in exactly the rows $R$ and columns $C$ and avoiding $\sigma$ equals the number of such fillings avoiding $\tau$.  In this case, we say that $\sigma \fs \tau$.
\end{definition}

It is easy to see that shape-Wilf-equivalence implies Wilf-equivalence by taking the $n \times n$ boards.  It is also easy to see that filling-shape-Wilf-equivalence implies shape-Wilf-equivalence since we can take $R = C = \varnothing$.

\subsection{Motivating example}
To motivate the following results,
we briefly outline a proof of the Wilf-equivalence
\[ 1342 \equiv 1432. \]
As this section is motivational only,
it may be skipped without loss of continuity, though we reuse the main idea of bijective swapping expained in this section in the proofs of Theorem \ref{thm:BijectionCons} and Theorem \ref{thm:BijectionVinc}.

We define a map $\Psi : S_n \to S_n$ as follows.
Given a permutation viewed as a standard filling
of an $n \times n$ Young tableau,
take each instance of $1342$ or $1432$,
and swap the two elements in the middle,
for example as in Figure~\ref{fig:swap}.
\begin{figure}[ht]
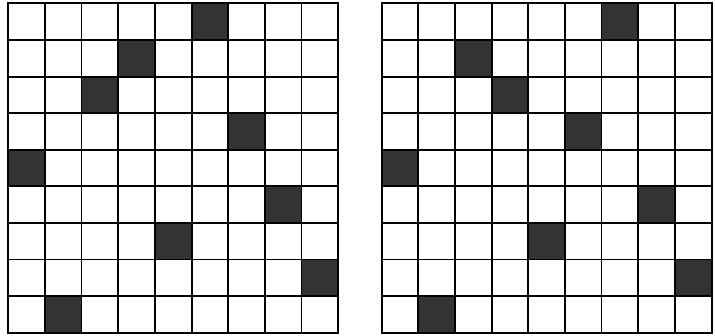

\centering
\small
\begin{ytableau}
\none & & & & & & *(black!80) & & & \\
\none & & & & *(black!80) & & & & & \\
\none & & & *(black!80) & & & & & & \\
\none & & & & & & & *(black!80) & & \\
\none & *(black!80) & & & & & & & & \\
\none & & & & & & & & *(black!80) & \\
\none & & & & & *(black!80) & & & & \\
\none & & & & & & & & & *(black!80) \\
\none & & *(black!80) & & & & & & & \\
\end{ytableau}
\begin{ytableau}
\none & & & & & & & *(black!80) & & \\
\none & & & *(black!80) & & & & & & \\
\none & & & & *(black!80) & & & & & \\
\none & & & & & & *(black!80) & & & \\
\none & *(black!80) & & & & & & & & \\
\none & & & & & & & & *(black!80) & \\
\none & & & & & *(black!80) & & & & \\
\none & & & & & & & & & *(black!80) \\
\none & & *(black!80) & & & & & & & \\
\end{ytableau}

\caption{An example of the bijection $\Psi$.
The consecutive patterns $1342$ and $1432$ appear starting at columns $2$ and $5$.}
\label{fig:swap}
\end{figure}

One can check that the pair of consecutive patterns
$\{1342, 1432\}$ has the property that any two instances
of either in a given permutation overlap in at most one place;
from this it follows that $\Psi$ is well-defined,
and an involution on $S_n$.
But $\Psi$ also maps $1342$-avoiding permutations
to $1432$-avoiding permutations and vice-versa,
which implies the Wilf-equivalence.

\subsection{Statement of results}
\label{sec:BijectionCons}
We now prove a family of filling shape-Wilf-equivalences for
consecutive and quasi-consecutive patterns,
formally stating the ideas in the previous section in their full generality.
This will give new proofs of all consecutive Wilf-equivalences of length $4$ and $5$, answer a filling-shape-Wilf-equivalence conjecture by Baxter, and help complete the Wilf-equivalence classification of quasi-consecutive patterns of length $5$.

\begin{theorem} \label{thm:BijectionCons}
    Suppose that 
	\[ \sigma = \sigma_1\sigma_2\cdots\sigma_i\sigma_{i+1}\cdots
		\sigma_{j}\sigma_{j+1}\cdots\sigma_k \]
	is a permutation where $1 \le i < j < k.$ 
	Also suppose that 
	\[ \tau = \tau_1\tau_2\cdots\tau_i\tau_{i+1}\cdots
		\tau_j\tau_{j+1}\cdots\tau_k \]
	where $\tau_{i+1}, \dots, \tau_j$ is a permutation of
	$\sigma_{i+1}, \dots, \sigma_j$, and $\sigma_x = \tau_x$
	for $x \in [1,i] \cup [j+1, k]$.
	Finally, suppose that the nonoverlapping criteria
	\begin{align*}
		\sigma_1\cdots\sigma_z &\not\sim \sigma_{k-z+1}\cdots\sigma_k \\
		\sigma_1\cdots\sigma_z &\not\sim \tau_{k-z+1}\cdots\tau_k \\
		\tau_1\cdots\tau_z &\not\sim \sigma_{k-z+1}\cdots\sigma_k \\
		\tau_1\cdots\tau_z &\not\sim \tau_{k-z+1}\cdots\tau_k
	\end{align*}
	all hold for $z > \min(i, k-j)$ and $z \neq  k$.
	Then, \[ \sigma \equiv_{\mathrm{fs}} \tau. \]
\end{theorem}

\begin{proof}
    Fix a Young diagram and a set of rows and columns which are avoided.  We present a bijection mapping fillings of containing $\sigma$ (call the set $T$) to fillings containing $\tau$ (call the set $T'$.)  Note $T$ and $T'$ need not be disjoint.
	
	For each such permutation in $S_n$, draw a box of width $k$ around every consecutive occurrence of $\sigma$ or $\tau$.  Because of the four equivalences not occurring, no two boxes can intersect in more than $\min(i, k-j)$ places.  Also since $i < j$ we cannot have three boxes pairwise intersecting.
    
	Consider the map \[ \Psi : S_n \to S_n \]
	defined by permuting every boxed $\sigma$ permutation to make it a $\tau$ permutation and vice versa.  Because boxes do not intersect in more than $\min(i, k-j)$ places, the order of permutation does not matter.  This also means that $\Psi$ sends $T$ to $T'$ and $T'$ to $T$, and that $\Psi$ is a well defined map.
    
    Also, our four equivalence restrictions prevent new $\sigma$ or $\tau$ permutations from appearing because a new occurrence of $\sigma$ or $\tau$ would need to have at least one element in a block of $k$ elements changed.  Thus, it would need to intersect with a part of a boxed segment in at least one of positions $i+1, \ldots, j$, a clear contradiction.  This means for $\pi \in T \cup T'$ the boxes around $\pi$ will be the same as the boxes around $\Psi(\pi)$.  This means $\Psi$ is an involution, as every box will permute from $\sigma$ to $\tau$ to $\sigma$ again, or vice versa.
	
	Thus, $\Psi$ induces a bijection from $T$ to $T'$, so we are done.
\end{proof}

\begin{corollary}
    \label{BaxterShapeConjecture}
    The following filling-shape-Wilf equivalences hold:
    \begin{itemize}
		\item $1342 \fs 1432$
		\item $2341 \fs 2431$
		\item $4213 \fs 4123$
		\item $3214 \fs 3124$.
	\end{itemize}
\end{corollary}

\begin{corollary}
	\label{cor:known5}
    The following filling-shape-Wilf-equivalences hold:
	\begin{itemize}
		\item $13452 \fs 13542 \fs 14352 \fs 14532 \fs 15342 \fs 15432$
		\item $12453 \fs 12543$
		\item $24153 \fs 25143$.
	\end{itemize}
\end{corollary}

Corollary \ref{BaxterShapeConjecture} proves all four parts of Conjecture $3$ in \cite{Baxter2013} and represents all nontrivial Wilf-equivalences of consecutive patterns of length 4, first proven in \cite{ElizaldeNoy2003}.  Corollary \ref{cor:known5} represents all Wilf-equivalences of consecutive patterns of length 5, first proven in \cite{Nakamura2011}.  Thus, our theorem provides a simple bijective proof of the consecutive cases for length 4 and 5 patterns. Our theorem is not sufficient for length 6 patterns, however, which was first proven in \cite{ElizaldeNoy2012}.

We can extend Theorem \ref{thm:BijectionCons} to
permutations of the form $\sigma_1 \- \sigma_2 \cdots \sigma_{k+1}$.

\begin{theorem}  \label{thm:BijectionVinc}
    Suppose that 
	\[ \sigma = \sigma_1 \- \sigma_2\cdots\sigma_{i+1}
		\sigma_{i+2}\cdots\sigma_{j+1}
		\sigma_{j+2}\cdots\sigma_{k+1}\]
	is a permutation where $1 \le i < j < k$.
	Also suppose that 
	\[ \tau = \tau_1 \- \tau_2\cdots \tau_{i+1}\tau_{i+2}\cdots\tau_{j+1}\tau_{j+2}\cdots\tau_{k+1},\]
	where $\tau_{i+2}, \dots, \tau_{j+1}$ is a permutation of
	$\sigma_{i+2}, \dots, \sigma_{j+1}$, and $\sigma_x = \tau_x$
	for $x \in [1,i+1] \cup [j+2, k+1]$.
	Finally, suppose that the nonoverlapping criteria
		\begin{align*}
		\sigma_2\cdots\sigma_{z+1} &\not\sim \sigma_{k-z+2}\cdots\sigma_{k+1} \\
		\sigma_2\cdots\sigma_{z+1} &\not\sim \tau_{k-z+2}\cdots\tau_{k+1} \\
		\tau_2\cdots\tau_{z+1} &\not\sim \sigma_{k-z+2}\cdots\sigma_{k+1} \\
		\tau_2\cdots\tau_{z+1} &\not\sim \tau_{k-z+2}\cdots\tau_{k+1}
	\end{align*}
	all hold for $z > \min(i, k-j)$, $z \neq  k$.
	Then, \[ \sigma \equiv_{\mathrm{fs}} \tau.\]
\end{theorem}

\begin{proof}
    The argument is quite similar to that given in Theorem \ref{thm:BijectionCons}.  Again, let $T$ be the set of $\sigma_1\-\sigma_2\cdots\sigma_{k+1}$-containing permutations and let $T'$ be the set of $\tau_1\-\tau_2\cdots\tau_{k+1}$-containing permutations.  Let a map $\Psi$ from $S_n$ to itself in a similar way as in Theorem \ref{thm:BijectionCons}, but now we draw boxes of width $k$ around every consecutive occurrence of $\sigma_2 \cdots \sigma_{k+1}$ and $\tau_2 \cdots \tau_{k+1},$ and then permute every boxed $\sigma_2 \cdots \sigma_{k+1}$ to make it a $\tau_2 \cdots \tau_{k+1}$ permutation and vice versa.  For the same reason as in Theorem \ref{thm:BijectionCons}, $\Psi$ is an involution.  We show that $\Psi$ bijects elements in $T$ containing $\sigma$ to elements in $T'$ containing $\tau.$
    
	To see why, suppose that $\pi$ in $T$ contains $\sigma$.  Suppose that 
	\[ \pi_s \pi_t\pi_{t+1}\cdots\pi_{t+k-1} \sim \sigma_1\cdots\sigma_{k+1},\]
	where $1 \le s \le t \le n-k+1$ and the $s$th and $t, (t+1), \ldots, (t+k-1)$th columns all contain the rows corresponding to $\pi_s, \pi_t, \ldots, \pi_{t+k-1}$.  Then, $\Psi$ sends $\pi_t\pi_{t+1}\cdots\pi_{t+k-1}$ to a permutation which is order isomorphic to $\tau_1\cdots\tau_k$, and either $\pi_s$  maps to itself or the value corresponding to $\pi_s$ would move to a position still before position $t$, say column $s'$.  However, since the column now with $\pi_s$ is before the $t$th column, and since the $t$th column contains the rows corresponding to $\pi_s, \pi_t, \ldots, \pi_{t+k-1}$, so will column $s'$.  Therefore, if $\pi$ contains the vincular pattern $\sigma,$ then $\Psi(\pi)$ will contain the vincular pattern $\tau$, and for the same reason, if $\pi$ contains the vincular pattern $\tau$, then $\Psi(\pi)$ will contain the vincular pattern $\sigma$.  As $\Psi$ is an involution, we are done.
\end{proof}

\begin{corollary} \label{cor:Bijective5quasi}
    The following equivalences hold:
	\begin{enumerate}
		\item[(B)] $1453\-2 \equiv 1543\-2$
		\item[(G)] $3125\-4 \equiv 3215\-4$.
	\end{enumerate}
\end{corollary}

These hold due to the complement operation and since filling-shape-Wilf-equivalence implies Wilf-equivalence.

\section{A family of inductive equivalences}
\label{sec:sandcastle}

In this section we prove one main general result for quasi-consecutive patterns by applying an inductive and bijective argument.  Baxter and Shattuck \cite{BaxterShattuckMainPaper} prove the following theorem.

\begin{theorem}[Theorem 6 of \cite{BaxterShattuckMainPaper}]
	\label{thm:BaxterShattuckThm6}
	Consider the quasi-consecutive vincular pattern 
	$\sigma = \sigma_1\sigma_2\cdots\sigma_{k} \- \sigma_{k+1}$. 
	Assume \[ \sigma_1<\cdots<\sigma_i = k > \sigma_{i+1}>\cdots>\sigma_k
	\quad\text{and}\quad \sigma_{k+1} = k+1. \]
	Let $\tau = \sigma_1\cdots\sigma_{i-1}\sigma_{k+1}\sigma_{i+1}\cdots\sigma_k\-\sigma_i$.
	Then $\sigma \equiv \tau.$
\end{theorem}

We prove a strong extension of this result by establishing a weaker but still sufficient condition as to when $\sigma_{k+1}$ and $\sigma_i$ can be switched, given $\sigma_i + 1 = \sigma_{k+1}.$

\begin{defn} \label{def:vAvoidance}
	For a vincular pattern $\sigma = \sigma_1\sigma_2\cdots\sigma_{k}\-\sigma_{k+1},$ we define the set $S_n(\sigma)[v]$ where $v \in [n]^r$ for some $0 \le r \le n$ as the subset of $S_n(\sigma)$ such that $\pi \in S_n$ is in $S_n(\sigma)[v]$ if and only if
	\begin{enumerate}
		\ii $\pi$ avoids $\sigma$, and
		\ii $\pi_i = v_i$ for all $1 \le i \le r$.
	\end{enumerate}
	We define $T_n(\sigma)[v]$ as the cardinality of $S_n(\sigma)[v]$.  Note that if $v$ is the empty string, $S_n(\sigma)[v] = S_n(\sigma).$
\end{defn}

\begin{defn}
    For a vincular pattern $\sigma = \sigma_1\sigma_2\cdots\sigma_{k}\-\sigma_{k+1},$ we define the set $S_n(\sigma)[x, i, w]$ where $0 \le x < n-k$, $1 \le i \le k$ and $w \in [n]^{x+k-1}$ as the subset of $S_n(\sigma)$ such that $\pi \in S_n$ is in $S_n(\sigma)[x, i, w]$ if and only if
	\begin{enumerate}
		\ii $\pi$ avoids $\sigma$, and
		\ii $\pi$ first contains the consecutive pattern $\sigma' = \sigma_1\sigma_2\cdots\sigma_{k}$ at positions $x+1, \dots, x+k$, and
		\ii $w = \pi_1\cdots\pi_{x+i-1}\pi_{x+i+1}\cdots\pi_{x+k}$.
	\end{enumerate}
	We define $T_n(\sigma)[x, i, w]$ as the cardinality of $S_n(\sigma)[x, i, w]$.
\end{defn}

\begin{example}
	We have $S_5(12\-3)[2, 1, 524] = \{52143\}.$
	To spell out the conditions for $\pi \in S_5(12\-3)[2, 1, 524]$,
	we need to have $\pi_1\pi_2\pi_4 = 524$,
	$\pi_3\pi_4$ must be the first occurrence of the consecutive pattern $12$,
	and $\pi$ must avoid $12\-3$.
	The permutation $\pi = 52143$ is the only one.
\end{example}

\begin{lemma} \label{lem:partition}
	Let $\sigma = \sigma_1\sigma_2\cdots\sigma_{k}\-\sigma_{k+1}$ and  $\tau = \tau_1\tau_2\cdots\tau_{k}\-\tau_{k+1}$ be vincular patterns.  Suppose that some fixed $n, i,$ and for all $0 \le x < n-k,$ and $w \in [n]^{x+k-1},$ $T_n(\sigma)[x, i, w] = T_n(\tau)[x, i, w].$  Then, for any $0 \le r < i, v \in [n]^r,$ we have that $T_n(\sigma)[v] = T_n(\tau)[v].$
\end{lemma}

\begin{proof}
	Let $\sigma' = \sigma_1 \cdots \sigma_k$ and $\tau' = \tau_1 \cdots \tau_k.$  We can partition $S_n(\sigma)[v]$ and $S_n(\tau)[v]$ into appropriate sets.  Note that any permutation $\pi$ is in $S_n(\sigma)[v]$ if and only if it is in $S_n(\sigma)[x, i, w]$ for some $x < n-k$ and some $w \in [n]^{x+k-1},$ where $w_1 \cdots w_{i-1} = v_1 \cdots v_{i-1},$ or $\pi \in S_n(\sigma'),$ or $\pi$ contains $\sigma',$ but only at positions $n-k+1, ..., n,$ i.e.\ at the very end of the permutation.  This is clearly a partitioning, since no permutation can be in multiple sets $S_n(\sigma)[x, i, w],$ and if $\pi$ avoids $\sigma'$ or only contains $\sigma'$ at positions $n-k+1, \dots, n$, it is not in any $S_n(\sigma)[x, i, w].$  We can do the same partitioning for all $\pi \in S_n(\tau)[v]$ also.  However, we know that $S_n(\sigma)[x, i, w] = S_n(\tau)[x, i, w]$ for all $0 \le x < n-k, w \in [n]^{x+k-1}.$  Also, since $\sigma' \sim \tau',$ the number of permutations avoiding $\sigma'$ equals the number of permutations avoiding $\tau'$ and the number of permutations only containing $\sigma'$ at positions $n-k+1, \dots, n$ equals the number of permutations only containing $\tau'$ at positions $n-k+1, \dots, n.$  The result follows by summing over the partitions.
\end{proof}

\begin{lemma}
	\label{lem:meow}
    Let $k \ge 3$ and $\sigma = \sigma_1\sigma_2\cdots\sigma_{k}\-\sigma_{k+1}$
	be a vincular pattern such that $\sigma_{k+1} = \sigma_i + 1$
	for some $2 \le i \le k-1$.
	Suppose that $\tau = \sigma_1\cdots\sigma_{i-1}\sigma_{k+1}\sigma_{i+1}\dots\sigma_k\-\sigma_i$.
	Also, suppose that the nonoverlapping criteria
	\[ \sigma_1\dots\sigma_z \not\sim \sigma_{k-z+1}\dots\sigma_k \] 
	hold for all $z \neq k, z \ge \min(i, k-i+1)$.
    Then
	\[ T_n(\sigma)[x, i, w] = T_n(\tau)[x, i, w]. \]
    
\end{lemma}

\begin{proof}
    We start by defining a subset $Y \in [k]$ consisting of the numbers $1 \le Y_1 < \dots < Y_{\sigma_i-1} \le k$ such that $\sigma_i > \sigma_{Y_j}$ for any $1 \le j \le \sigma_i-1$ (note if $\sigma_i = 1$ then $Y = \emptyset$).  Define $Z$ similarly as $[k] - Y - \sigma_i,$ i.e.\ the numbers $1 \le Z_1 < \dots < Z_{k-\sigma_i} \le k$ such that $\sigma_i < \sigma_{Z_j}$ for any  $1 \le j \le k-\sigma_i.$  Note that $\tau_i > \tau_{Y_j}$ for all $1 \le j \le \sigma_i-1$ and $\tau_i < \tau_{Z_j}$ for all $1 \le j \le k-\sigma_i$.
    
    We prove this by strong induction on the length $n$ of the permutation, starting with base case $1 \le n \le k+1.$  If $1 \le n \le k$ then the result is trivial.  If $n = k+1$ then it is obvious that $x = 0$ since $0 \le x < n-k = 1.$  Also, assume that $w_1 \cdots w_{k-1} \sim \sigma_1 \cdots \sigma_{i-1}\sigma_{i+1} \cdots \sigma_k,$ or else it is clear that $T_n(\sigma)[0, i, w] = T_n(\tau)[0, i, w] = 0.$  Now, any $\pi \in S_n(\sigma)[x, i, w] \cup S_n(\tau)[x, i, w]$ must have $\pi_1, \dots, \pi_{i-1}, \pi_{i+1}, \dots, \pi_k$ fixed due to $w$, leaving only two remaining elements: $a, b \in [k+1] \backslash \{\pi_1, \dots, \pi_{i-1}, \pi_{i+1}, \dots, \pi_k\}.$  Therefore, there are at most two elements in $S_n(\sigma)[x, i, w] \cup S_n(\tau)[x, i, w]$ since we can choose either $\pi_i = a, \pi_{k+1} = b$ or $\pi_i = b, \pi_{k+1} = a$.  If both $a$ and $b$ are larger than all elements of the form $\pi_{Y_j}$ but smaller than all elements of the form $\pi_{Z_j},$ then exactly one assignment of $a, b$ to $\pi_i, \pi_{k+1}$ will avoid $\sigma$ and the other choice will avoid $\tau,$ but both will contain the consecutive pattern $\sigma' \sim \tau'$ at positions $1, \dots, k.$  In this case, $T_n(\sigma)[x, i, w] = T_n(\tau)[x, i, w] = 1.$  If exactly one of $a, b$ (assume WLOG $a$) is larger than all elements $\pi_{Y_j}$ but smaller than all elements $\pi_{Z_j}$ then any element in $S_n(\sigma)[x, i, w]$ or in $S_n(\tau)[x, i, w]$ must have $\pi_i = a$ so that $\pi_1 \cdots \pi_k \sim \sigma_1 \cdots \sigma_k \sim \tau_1 \cdots \tau_k.$  But then this permutation will clearly avoid the vincular patterns $\sigma$ and $\tau,$ so again $T_n(\sigma)[x, i, w] = T_n(\tau)[x, i, w] = 1.$  Finally, if neither $a$ nor $b$ is larger than all elements $\pi_{Y_j}$ but smaller than all elements $\pi_{Z_j}$, then there is no permutation such that $\pi_1 \cdots \pi_k \sim \sigma_1 \cdots \sigma_k \sim \tau_1 \cdots \tau_k,$ so $T_n(\sigma)[x, i, w] = T_n(\tau)[x, i, w] = 0.$  This proves the base case.	

    Now, consider $n > k+1.$  Assume that $w_{x+1} \cdots w_{x+i-1} w_{x+i+1} \cdots w_{x+n} \sim \sigma_1 \cdots \sigma_{i-1} \sigma_{i+1} \cdots \sigma_n,$ or else $T_n(\sigma)[x, i, w] = T_n(\tau)[x, i, w] = 0.$  If $\pi \in S_n(\sigma)[x, i, w] \cup S_n(\tau)[x, i, w],$ then $\pi_{x+i}$ must be larger than the elements in positions $x+Y_j$ and smaller than the elements in positions $x+Z_j.$  Consider the elements among $[n] - \{w_1, \dots, w_{x+k-1}\}$ (recall that $w_1 \cdots w_{x+k-1} = \pi_1 \cdots \pi_{x+i-1} \pi_{x+i+1} \cdots \pi_{x+k}$ for any $\pi \in S_n(\sigma)[x, i, w] \cup S_n(\tau)[x, i, w]$) that are larger than all elements in positions $x+Y_j$ and smaller than all elements in positions $x+Z_j.$  Then, a permutation $\pi$ avoids $\sigma$ if and only if the largest possible element is chosen (so that a $\sigma$ permutation doesn't appear where $\sigma'$ starts at position $x+1$) and the reduction of $\pi_{x+2} \cdots \pi_n$ avoids $\sigma$.  However, because $\sigma_1\cdots\sigma_z \sim \sigma_{k-z+1}\cdots\sigma_k$ only if $i = k$ or if $i < k-z+1,$ (equivalent to $z < k-i+1$), a later $\sigma'$ permutation cannot start until at least position $x+i+1$.  Also, because $\sigma_1\cdots\sigma_z \sim \sigma_{k-z+1}\cdots\sigma_k$ only if $i = k$ or if $z < i$, a later $\sigma'$ cannot start until at least position $x+k-i+2,$ which means it has at most $i-1$ elements overlapping with positions $1$ to $x+k$.  Thus, a permutation $\pi$ such that $w_1 \cdots w_{x+k-1} = \pi_1 \cdots \pi_{x+i-1} \pi_{x+i+1} \cdots \pi_{x+k}$ avoids $\sigma$ if and only if the largest possible element less than all $\pi_{x+Z_j}$ and greater than all $\pi_{x+Y_j}$ is chosen for $\pi_{x+i}$ and the reduction of $\pi_{\max(x+k-i+2, x+i+1)}, \dots, \pi_n$ avoids $\sigma$.  Likewise, a $\pi$ such that $w_1 \cdots w_{x+k-1} = \pi_1 \cdots \pi_{x+i-1} \pi_{x+i+1} \cdots \pi_{x+k}$ avoids $\tau$ if and only if the smallest possible element less than all $\pi_{x+Z_j}$ and greater than all $\pi_{x+Y_j}$ is chosen for $\pi_{x+i}$ and the reduction of $\pi_{\max(x+k-i+2, x+i+1)}, \dots, \pi_n$ avoids $\tau$ for the same reason.

	Let $q_s$ and $q_l$ be the smallest and largest possible values of $\pi_{x+i},$ respectively. Note that $\{\pi_1, \dots, \pi_{x+i-1}, \pi_{x+i+1}, \dots, \pi_{x+k}\}$ have been chosen and that $q_s > S_{x+j}$ if and only if $q_l > S_{x+j}$ for $1 \le j \le k, j \neq i.$  Thus, if we reduce $\pi_{\max(x+k-i+2, x+i+1)}, \dots, \pi_{n},$ the resulting reduction of $\pi_{\max(x+k-i+2, x+i+1)}, \dots, \pi_{x+k}$ is always fixed, which is at most $i-1$ elements.  It also does not depend on whether $q_s$ or $q_l$ is chosen.  By our induction hypothesis we have that 
	\[ T_{n-\max(x+k-i+2, x+i+1)+1}(\sigma)[x', i, w'] = T_{n-\max(x+k-i+2, x+i+1)+1}(\tau)[x', i, w'] \]
	for any feasible $x', w'.$  Therefore, by Lemma \ref{lem:partition}, we are done.
\end{proof}

\begin{theorem}
	\label{thm:sandcastle}
	Retaining the setting of Lemma~\ref{lem:meow}, we have $\sigma \equiv \tau$.
\end{theorem}

\begin{proof}
    The theorem follows from Lemmas \ref{lem:partition} and \ref{lem:meow} after setting $v$ to be the empty string in Lemma \ref{lem:partition} .
\end{proof}

    

We now present the new corollaries that can be directly obtained from applying Theorem \ref{thm:sandcastle} to the case where $k = 4$.

\begin{corollary} \label{cor:Inductive5quasi}
    The following equivalences hold:
	\begin{enumerate}
		\item[(A)] $1254\-3 \equiv 1354\-2$
		\item[(C)] $2135\-4 \equiv 2145\-3$
        \item[(Fa)] $1243\-5 \equiv 1253\-4$
        \item[(Fb)] $3245\-1 \equiv 3145\-2$
		\item[(La)] $1342\-5 \equiv 1352\-4 \equiv 1452\-3$
		\item[(Lb)] $1432\-5 \equiv 1532\-4 \equiv 1542\-3$
		\item[(Lc)] $2431\-5 \equiv 2531\-4 \equiv 2541\-3$
		\item[(Ld)] $2341\-5 \equiv 2351\-4 \equiv 2451\-3$.
	\end{enumerate}
\end{corollary}

It should be noted that (Fa) and the first congruence in (La), (Lb), (Lc), and (Ld) are in fact corollaries of Theorem \ref{thm:BaxterShattuckThm6}. 

\section{Proof of equivalence M, $2153\-4 \equiv 3154\-2$}
\label{sec:M}
In the section we prove that $2153\-4 \equiv 3154\-2$.

\subsection{Defining the recursion}

Similar to Definition \ref{def:vAvoidance}, for a vincular pattern $\sigma,$ define $T_n(\sigma)$ to equal $|S_n(\sigma)|.$  
Define $T_n(\sigma)[k]$ denote the number of permutations $\pi \in S_n$ that avoid $\sigma$ and such that $\pi_1 = k$, and similarly, define $T_n(\sigma)[k, \ell]$ denote the number of permutations
avoiding $\sigma$ such that $\pi_1 = k, \pi_2 = \ell$.

For the remainder of this section, we will fix
\begin{align*}
	\sigma &= 2153\-4 \\
	\tau &= 3154\-2.
\end{align*}
For convenience, Appendix~\ref{sec:SHLchart}
lists several values of $T_n(\sigma)[k, \ell]$ and $T_n(\tau)[k, \ell]$.

We prove the following recursion.
\begin{lemma}
	\label{lem:2153-4Rec}
	For any $1 \le k, \ell \le n$ we have
	\[ T_n(\sigma)[k, \ell] = 
    \begin{cases} 
      T_{n-1}(\sigma)[\ell-1] & k < \ell \\
	  0 & k = \ell \\
      T_{n-1}(\sigma)[\ell] - 
	  \displaystyle\sum_{\substack{i \ge j+2 \\ j \ge k-1}} T_{n-2}(\sigma)[i, j] & k > \ell.
   \end{cases}
   \]
\end{lemma}

\begin{proof}
	Obviously $T_n(\sigma)[k, \ell] = 0$ if $k = \ell$, since
	permutations cannot contain repeated elements.

	We show that a permutation avoids $2153\-4$ if and only if both the last $n-1$ elements avoid $2153\-4$ and, if $\pi_1\pi_2\pi_3\pi_4 \sim 2143,$ then $\pi_3$ and $\pi_4$ are consecutive.  To see why, note that avoiding $2153\-4$ clearly means that the last $n-1$ elements avoid $2153\-4$.  Also, if a permutation $\pi$ avoids $2153\-4$ but begins with four elements order isomorphic to $2143$ and $\pi_3 > \pi_4+1,$ then $\pi_4+1$ appears later, causing an instance of $2153\-4.$  Also, if a $2153\-4$ occurs, it either occurs in the last $n-1$ elements or the $2153$ part occurs at the beginning, which means that $\pi_1\pi_2\pi_3\pi_4 \sim 2143.$  But then $\pi_3 = \pi_4+1$ means that $2153\-4$ cannot occur with the $21534$ part occurring at the beginning.  This proves our claim in the first sentence.
    
    Now, this means that a permutation that begins with $k, \ell$ avoids $2153\-4$ only if the reduction of the last $n-1$ elements avoids $2153\-4.$  Since there is clearly a unique map between reductions of the last $n-1$ elements and the permutations, there are $T_{n-1}(\sigma)[l-1]$ such permutations if $k<\ell$ and $T_{n-1}(\sigma)[\ell]$ such permutations otherwise, as it depends on what $\ell$ reduces to.
    
    However, we must consider the possibility that the last $n-1$ elements avoid $2153\-4$ but $\pi_1\pi_2\pi_3\pi_4 \sim2143$ when $\pi_3$ and $\pi_4$ are not consecutive.  This never happens if $k < \ell$, so we have proven our recursion for this case.  Now, assume $k > \ell.$  Note that if $\pi_1\pi_2\pi_3\pi_4 \sim 2143$, then $\pi_2 < \pi_3$ implies that the last $n-1$ elements avoiding $2153\-4$ is equivalent to the last $n-2$ elements avoiding $2153\-4.$ Now, we sum up the ways this can occur given a fixed $\pi_3$ and $\pi_4.$  Since $\pi_3, \pi_4 > k \ge 2,$ we can say $\pi_3 = i+2, \pi_4 = j+2.$  Then, $\pi_3, \pi_4$ reduce to $i, j$.  We now just need to count the number of times when $i+2 > j+2 > k$. Observe that $i+2 \ge (j+2)+2$ as they cannot be consecutive, and the reduction of the last $n-2$ elements avoids $2153\-4.$  Based on definition, this clearly equals $T_{n-2}(\sigma)[i, j].$  Summing over all $i, j,$ where $i \ge j+2$ and $j+2 > k,$ or $j \ge k-1,$ we get our desired result and we are done with our recursion.
\end{proof}

We now prove a similar recursion for permutations avoiding $3154\-2.$ 

\begin{lemma}
	\label{lem:3154-2Rec}
	For $1 \le k, \ell \le n$ we have
    \[
		T_n(\tau)[k, \ell] = 
		\begin{cases} 
		  T_{n-1}(\tau)[\ell-1] & k < \ell \\
		  0 & k = \ell \\
		  T_{n-1}(\tau)[\ell] & k = \ell+1 \\
		  T_{n-1}(\tau)[\ell]
		  - \displaystyle\sum_{\substack{i \ge j+1 \\ j \ge k-1}} T_{n-2}(\tau)[i, j] & k > \ell+1.
	   \end{cases}
   \]
\end{lemma}

\begin{proof}
	As before $T_n(\tau)[k, \ell] = 0$ if $k = \ell$, since
	permutations cannot contain repeated elements.

	A permutation avoids $3154\-2$ if and only if both the reduction of the last $n-1$ elements avoid $3154\-2$ and, if $\pi_1\pi_2\pi_3\pi_4 \sim 2143$, then $k = \ell+1$.  The proof is analogous to the first paragraph in the proof of Lemma~\ref{lem:2153-4Rec}.
    
    Now, this means that a permutation beginning with $k, \ell$ avoids $3154\-2$ only if the reduction of the last $n-1$ elements avoids $3154\-2.$  Since there is a natural bijection between reductions of the last $n-1$ elements and the permutations, there are $T_{n-1}(\tau)[\ell-1]$ such permutations if $k<\ell$ and $T_{n-1}(\tau)[\ell]$ such permutations otherwise, as it depends on what $\ell$ reduces to.
    
    However, we must consider the possibility that the last $n-1$ elements avoid $3154\-2$ but $\pi_1\pi_2\pi_3\pi_4 \sim 2143$ when $\pi_1$ and $\pi_2$ are not consecutive.  This never happens if $k < \ell$ or if $k = \ell+1$, so we have proven our recursion for this case.  Now, assume $k > \ell+1.$  Note that if $\pi_1\pi_2\pi_3\pi_4 \sim 2143$, then $\pi_2 < \pi_3$ implies that the last $n-1$ elements avoiding $3154\-2$ is equivalent to the last $n-2$ elements avoiding $3154\-2.$ Now, we sum up the ways this can occur given a fixed $\pi_3$ and $\pi_4.$  Since $\pi_3, \pi_4 > k \ge 2,$ we again say $\pi_3 = i+2, \pi_4 = j+2.$  Then, $\pi_3, \pi_4$ reduce to $i, j$.  We now just need to count the number of times when $i+2 > j+2 > k$, and the reduction of the last $n-2$ elements avoids $3154\-2.$  Based on definition, this clearly equals $T_{n-2}(\sigma)[i, j].$  Summing over all $i, j,$ where $i \ge j+2$ and $j+2 > k,$ or $j \ge k-1,$ we get our desired result and we are done with our recursion.
\end{proof}


\subsection{Proof of equivalence}
Now that we have proved both Lemmas~\ref{lem:2153-4Rec} and \ref{lem:3154-2Rec},
we also make the following remark.
\begin{proposition}
	\label{prop:easyH}
	For any $n$, we have
	\begin{enumerate}[(a)]
		\item $T_n(\sigma)[1] = T_n(\sigma)[n-1] = T_n(\sigma)[n] = T_{n-1}(\sigma)$.
		\item $T_n(\tau)[1] = T_n(\tau)[2] = T_n(\tau)[n] = T_{n-1}(\tau)$.
		\item $T_n(\sigma)[n,\ell] = T_{n-1}(\sigma)[\ell]$.
	\end{enumerate}
\end{proposition}
\begin{proof}
	In all cases, the first letter cannot
	be the first letter of the pattern in question.
\end{proof}

We are now ready to prove the main result.
\begin{theorem}
	\label{thm:H}
	For $n \ge 5$ and $k \neq \ell$ the following statements hold.
	\begin{enumerate}[(a)]
		\item For $k < \ell < n$,
			we have $T_n(\sigma)[k, \ell] = T_{n-1}(\tau)[\ell]$.
			For $k < \ell = n$ we have
			we have $T_n(\sigma)[k, \ell] = T_{n-1}(\tau)[1]$.
		\item For $n = k$ and $\ell \neq n-1$,
			we have $T_n(\sigma)[n, \ell] = T_{n-1}(\tau)[\ell+1]$.
			Also, $T_n(\sigma)[n, n-1] = T_{n-1}(\tau)[n-1]$.
		\item For $n > k > \ell+1$,
			we have $T_n(\sigma)[k, \ell] = T_n(\tau)[k+1, \ell+1]$.
		\item For any $1 \le k < n$,
			we have $T_n(\sigma)[k] = T_n(\tau)[k+1]$ and
			$T_n(\sigma)[n]= T_n(\tau)[1]$.
		\item $T_n(\sigma) = T_n(\tau)$.
	\end{enumerate}
	Thus $2153\-4 \equiv 3154\-2$.
\end{theorem}
\begin{proof}
	We prove all parts simultaneously by induction on $n$.
	The base cases $n=5$ and $n=6$ can be checked manually;
	see Appendix~\ref{sec:SHLchart}.

	For part (a), when $\ell \neq n$ we have
	$T_n(\sigma)[k,\ell] = T_{n-1}(\sigma)[\ell-1] = T_{n-1}(\tau)[\ell]$,
	the first equality from the recursion and the second from (d).
	If $\ell = n$ we instead get $T_{n-1}(\sigma)[n-1] = T_{n-1}(\tau)[1]$
	again from (d).

	When $k=n$ the sum in the recursion is empty and
	$T_n(\sigma)[n, \ell] = T_{n-1}(\sigma)[\ell]$,
	so part (b) follows in the same way as part (a)
	if $\ell \neq n-1$.
	As for $k = n, \ell = n-1$,
	note that both $T_n(\sigma)[n,n-1]$
	and $T_{n-1}(\tau)[n-1]$ are equal to
	$T_{n-2}(\sigma) = T_{n-2}(\tau)$
	by Proposition~\ref{prop:easyH}.

	For part (c), note that
	\begin{align*}
		T_n(\sigma)[k, \ell] &= T_{n-1}(\sigma)[\ell] - \sum_{\substack{i \ge j+2 \\ j \ge k-1}} T_{n-2}(\sigma)[i,j] \\
		&= T_{n-1}(\sigma)[\ell] - \sum_{\substack{i \ge j+2 \\ i \neq n-2 \\ j \ge k-1}} T_{n-2}(\sigma)[i,j]
			- \sum_{n-4 \ge j \ge k-1} T_{n-2}(\sigma) [n-2,j]. \\
		\intertext{By the inductive hypothesis, we now obtain}
		T_n(\sigma)[k, \ell] 
		&= T_{n-1}(\tau)[\ell+1] - \sum_{\substack{i \ge j+2 \\ i \neq n-2 \\ j \ge k-1}} T_{n-2}(\tau)[i+1,j+1]
			- \sum_{n-4 \ge j \ge k-1} T_{n-3}(\sigma) [j] \\
		&= T_{n-1}(\tau)[\ell+1] - \sum_{\substack{i \ge j+2 \\ i \neq n-2 \\ j \ge k-1}} T_{n-2}(\tau)[i+1,j+1]
			- \sum_{n-4 \ge j \ge k-1} T_{n-3}(\tau) [j+1] \\
		&= T_{n-1}(\tau)[\ell+1] - \sum_{\substack{i \ge j+2 \\ i \neq n-2 \\ j \ge k-1}} T_{n-2}(\tau)[i+1,j+1]
			- \sum_{n-3 \ge j \ge k} T_{n-3}(\tau) [j] \\
		&= T_{n-1}(\tau)[\ell+1] - \sum_{\substack{i \ge j+2 \\ j \ge k}} T_{n-2}(\tau)[i,j]
			- \sum_{n-3 \ge j \ge k} T_{n-2}(\tau) [j+1, j] \\
		\intertext{Merging the sums,}
		T_n(\sigma)[k, \ell] 
		&= T_{n-1}(\tau)[\ell+1] - \sum_{\substack{i \ge j+1 \\ j \ge k}} T_{n-2}(\tau)[i,j]  \\
		&= T_{n-1}(\tau)[\ell+1] - \sum_{i > j \ge (k+1)-1} T_{n-2}(\tau)[i,j] \\
		&= T_n(\tau)[k+1, \ell+1].
	\end{align*}

	Finally we prove part (d).
	From Proposition~\ref{prop:easyH}, we have the equalities.
	\[
		T_n(\sigma)[1] = T_n(\tau)[2], \qquad
		T_n(\sigma)[n-1] = T_n(\tau)[n], \qquad
		T_n(\sigma)[n] = T_n(\tau)[1].
	\]
	This eliminates the cases $k=1, n-1, n$.

	So now assume $2 \le k \le n-2$;
	we wish to show that $T_n(\sigma)[k] = T_n(\tau)[k+1]$.
	Notice we have the equalities
	\begin{align*}
		T_n(\sigma)[k,1] &= T_n(\tau)[k+1, 2] \\
		T_n(\sigma)[k,2] &= T_n(\tau)[k+1, 3] \\
		&\vdotswithin= \\
		T_n(\sigma)[k,k-2] &= T_n(\tau)[k+1,k-1] \\
		T_n(\sigma)[k,k+1] &= T_{n-1}(\tau)[k+1] = T_n(\tau)[k+1,k+2] \\
		T_n(\sigma)[k,k+2] &= T_{n-1}(\tau)[k+2] = T_n(\tau)[k+1,k+3] \\
		T_n(\sigma)[k,k+3] &= T_{n-1}(\tau)[k+3] = T_n(\tau)[k+1,k+4] \\
		&\vdotswithin= \\
		T_n(\sigma)[k,n-1] &= T_{n-1}(\tau)[n-1] = T_n(\tau)[k+1,n].
	\end{align*}
	By matching together all these terms,
	we see that it only remains to show
	\[
		T_n(\sigma)[k,k-1] + T_n(\sigma)[k,n]
		= T_n(\tau)[k+1,1] + T_n(\tau)[k+1,k].
	\]

	Notice that by Lemma \ref{lem:2153-4Rec} and our induction hypothesis,
	\begin{align*}
		T_n(\sigma)[k,k-1] &= T_{n-1}(\sigma)[k-1] - \sum\limits_{\substack{i \ge j+2 \\ j \ge k-1}} T_{n-2}(\sigma)[i, j] \\
        &= T_{n-1}(\tau)[k] - \sum\limits_{\substack{n-3 \ge i \ge j+2 \\ j \ge k-1}} T_{n-2}(\sigma)[i, j] - \sum\limits_{n-4 \ge j \ge k-1} T_{n-2}(\sigma)[n-2, j] \\
        &= T_{n-1}(\tau)[k] - \sum\limits_{\substack{i \ge j+2 \\ j \ge k}} T_{n-2}(\tau)[i, j] - \sum\limits_{n-4 \ge j \ge k-1} T_{n-2}(\sigma)[n-2, j].
	\end{align*}
    However, for any $j \le n-4,$ $T_{n-2}(\sigma)[n-2, j] = T_{n-3}(\sigma)[j]$ by Proposition \ref{prop:easyH}, and by our induction hypothesis, this equals $T_{n-3}(\tau)[j+1]$, which in turn equals $T_{n-2}(\tau)[j+2, j+1]$ by our recursion.  Therefore, 
    \begin{align*}
    	T_n(\sigma)[k, k-1] &= T_{n-1}(\tau)[k] - \sum\limits_{\substack{i \ge j+2 \\ j \ge k}} T_{n-2}(\tau)[i, j] - \sum\limits_{n-4 \ge j \ge k-1} T_{n-2}(\tau)[j+2, j+1] \\
        &= T_{n-1}(\tau)[k] - \sum\limits_{i > j \ge k} T_{n-2}(\tau)[i, j].
    \end{align*}
	In light of $T_{n-1}(\tau)[k+1,k] = T_{n-1}(\tau)[k]$,
	it suffices then to show that
	\[ T_n(\sigma)[k,n] = T_n(\tau)[k+1,1] + \sum_{i>j\ge k} T_{n-2}(\tau)[i,j]. \]
	Finally, we note $T_n(\sigma)[k,n] = T_{n-1}(\tau)[1]$ by (a);
	substituting this above completes the proof of (d).

	Part (e) is immediate by part (d);
	this completes the induction,
	and the Wilf-equivalence follows.
\end{proof}

\newpage
\appendix
\section{Enumeration}
\label{sec:oeis}

The following tables list the number of avoiding permutations
in each equivalence class for $5 \le n \le 11$.

\begin{center}
	\centering
	\begin{tabular}{|crrrrrrr|}
		\hline
		& $n=5$ & $n=6$ & $n=7$ & $n=8$ & $n=9$ & $n=10$ & $n=11$ \\ \hline
A & $119$ & $704$ & $4838$ & $37864$ & $332456$ & $3236090$ & $34585138$ \\
B & $119$ & $704$ & $4838$ & $37864$ & $332476$ & $3236740$ & $34599992$ \\
C & $119$ & $704$ & $4838$ & $37865$ & $332477$ & $3236426$ & $34590177$ \\
D & $119$ & $704$ & $4838$ & $37866$ & $332516$ & $3237362$ & $34609120$ \\
E & $119$ & $704$ & $4838$ & $37867$ & $332537$ & $3237698$ & $34614147$ \\
F & $119$ & $704$ & $4838$ & $37868$ & $332546$ & $3237620$ & $34609332$ \\
G & $119$ & $704$ & $4838$ & $37868$ & $332558$ & $3238028$ & $34618998$ \\
H & $119$ & $704$ & $4838$ & $37870$ & $332606$ & $3238886$ & $34633106$ \\
I & $119$ & $704$ & $4838$ & $37870$ & $332606$ & $3238891$ & $34633233$ \\
J & $119$ & $704$ & $4838$ & $37870$ & $332620$ & $3239356$ & $34644052$ \\
K & $119$ & $704$ & $4838$ & $37870$ & $332622$ & $3239412$ & $34645000$ \\
L & $119$ & $704$ & $4838$ & $37874$ & $332696$ & $3240416$ & $34657116$ \\
M & $119$ & $704$ & $4838$ & $37875$ & $332731$ & $3241219$ & $34672985$ \\
N & $119$ & $704$ & $4838$ & $37879$ & $332845$ & $3243505$ & $34713895$ \\
O & $119$ & $704$ & $4839$ & $37886$ & $332821$ & $3241738$ & $34671733$ \\
P & $119$ & $704$ & $4839$ & $37887$ & $332842$ & $3242069$ & $34676639$ \\
Q & $119$ & $704$ & $4839$ & $37888$ & $332873$ & $3242738$ & $34689337$ \\
R & $119$ & $704$ & $4839$ & $37890$ & $332911$ & $3243252$ & $34695475$ \\
S & $119$ & $704$ & $4839$ & $37895$ & $333036$ & $3245568$ & $34734865$ \\
T & $119$ & $704$ & $4839$ & $37897$ & $333096$ & $3246798$ & $34757387$ \\
U & $119$ & $704$ & $4840$ & $37908$ & $333175$ & $3247036$ & $34750496$ \\
V & $119$ & $704$ & $4840$ & $37909$ & $333198$ & $3247430$ & $34756773$ \\
W & $119$ & $704$ & $4840$ & $37912$ & $333287$ & $3249227$ & $34789373$ \\
X & $119$ & $704$ & $4840$ & $37917$ & $333398$ & $3251054$ & $34817364$ \\
Y & $119$ & $704$ & $4840$ & $37918$ & $333474$ & $3253240$ & $34865094$ \\
Z & $119$ & $705$ & $4857$ & $38142$ & $336291$ & $3289057$ & $35337067$ \\
		\hline
	\end{tabular}
\end{center}

\newpage
\section{Table of values of $S_n(2153\-4)[k, \ell]$ and $S_n(3154\-2)[k, \ell]$}
\label{sec:SHLchart}
\small

\subsection{Recursion for $S_n(\sigma)[k, \ell]$, where $\sigma = 2153\-4$.}
\[ \begin{array}{|r|rrrrr|r|}
\hline
S_5(\sigma)[k, \ell]
& 1 & 2 & 3 & 4 & 5 & \Sigma \\ \hline
k= 1 & 0 & 6 & 6 & 6 & 6 & 24 \\
k= 2 & 5 & 0 & 6 & 6 & 6 & 23 \\
k= 3 & 6 & 6 & 0 & 6 & 6 & 24 \\
k= 4 & 6 & 6 & 6 & 0 & 6 & 24 \\
k= 5 & 6 & 6 & 6 & 6 & 0 & 24 \\
\hline
\Sigma & 23 & 24 & 24 & 24 & 24 & \mathbf{119}
\\ \hline
\end{array}
\]
\[ \begin{array}{|r|rrrrrr|r|}
\hline
S_6(\sigma)[k, \ell]
& 1 & 2 & 3 & 4 & 5 & 6 & \Sigma \\ \hline
k= 1 & 0 & 24 & 23 & 24 & 24 & 24 & 119 \\
k= 2 & 18 & 0 & 23 & 24 & 24 & 24 & 113 \\
k= 3 & 22 & 21 & 0 & 24 & 24 & 24 & 115 \\
k= 4 & 24 & 23 & 24 & 0 & 24 & 24 & 119 \\
k= 5 & 24 & 23 & 24 & 24 & 0 & 24 & 119 \\
k= 6 & 24 & 23 & 24 & 24 & 24 & 0 & 119 \\
\hline
\Sigma & 112 & 114 & 118 & 120 & 120 & 120 & \mathbf{704}
\\ \hline
\end{array}
\]
\[ \begin{array}{|r|rrrrrrr|r|}
\hline
S_7(\sigma)[k, \ell]
& 1 & 2 & 3 & 4 & 5 & 6 & 7 & \Sigma \\ \hline
k= 1 & 0 & 119 & 113 & 115 & 119 & 119 & 119 & 704 \\
k= 2 & 83 & 0 & 113 & 115 & 119 & 119 & 119 & 668 \\
k= 3 & 101 & 95 & 0 & 115 & 119 & 119 & 119 & 668 \\
k= 4 & 113 & 107 & 109 & 0 & 119 & 119 & 119 & 686 \\
k= 5 & 119 & 113 & 115 & 119 & 0 & 119 & 119 & 704 \\
k= 6 & 119 & 113 & 115 & 119 & 119 & 0 & 119 & 704 \\
k= 7 & 119 & 113 & 115 & 119 & 119 & 119 & 0 & 704 \\
\hline
\Sigma & 654 & 660 & 680 & 702 & 714 & 714 & 714 & \mathbf{4838}
\\ \hline
\end{array}
\]
\[ \begin{array}{|r|rrrrrrrr|r|}
\hline
S_8(\sigma)[k, \ell]
& 1 & 2 & 3 & 4 & 5 & 6 & 7 & 8 & \Sigma \\ \hline
k= 1 & 0 & 704 & 668 & 668 & 686 & 704 & 704 & 704 & 4838 \\
k= 2 & 469 & 0 & 668 & 668 & 686 & 704 & 704 & 704 & 4603 \\
k= 3 & 563 & 527 & 0 & 668 & 686 & 704 & 704 & 704 & 4556 \\
k= 4 & 632 & 596 & 596 & 0 & 686 & 704 & 704 & 704 & 4622 \\
k= 5 & 680 & 644 & 644 & 662 & 0 & 704 & 704 & 704 & 4742 \\
k= 6 & 704 & 668 & 668 & 686 & 704 & 0 & 704 & 704 & 4838 \\
k= 7 & 704 & 668 & 668 & 686 & 704 & 704 & 0 & 704 & 4838 \\
k= 8 & 704 & 668 & 668 & 686 & 704 & 704 & 704 & 0 & 4838 \\
\hline
\Sigma & 4456 & 4475 & 4580 & 4724 & 4856 & 4928 & 4928 & 4928 & \mathbf{37875}
\\ \hline
\end{array}
\]
{\scriptsize
\[ \begin{array}{|r|rrrrrrrrr|r|}
\hline
S_9(\sigma)[k, \ell]
& 1 & 2 & 3 & 4 & 5 & 6 & 7 & 8 & 9 & \Sigma \\ \hline
k= 1 & 0 & 4838 & 4603 & 4556 & 4622 & 4742 & 4838 & 4838 & 4838 & 37875 \\
k= 2 & 3119 & 0 & 4603 & 4556 & 4622 & 4742 & 4838 & 4838 & 4838 & 36156 \\
k= 3 & 3690 & 3455 & 0 & 4556 & 4622 & 4742 & 4838 & 4838 & 4838 & 35579 \\
k= 4 & 4136 & 3901 & 3854 & 0 & 4622 & 4742 & 4838 & 4838 & 4838 & 35769 \\
k= 5 & 4481 & 4246 & 4199 & 4265 & 0 & 4742 & 4838 & 4838 & 4838 & 36447 \\
k= 6 & 4719 & 4484 & 4437 & 4503 & 4623 & 0 & 4838 & 4838 & 4838 & 37280 \\
k= 7 & 4838 & 4603 & 4556 & 4622 & 4742 & 4838 & 0 & 4838 & 4838 & 37875 \\
k= 8 & 4838 & 4603 & 4556 & 4622 & 4742 & 4838 & 4838 & 0 & 4838 & 37875 \\
k= 9 & 4838 & 4603 & 4556 & 4622 & 4742 & 4838 & 4838 & 4838 & 0 & 37875 \\
\hline
\Sigma & 34659 & 34733 & 35364 & 36302 & 37337 & 38224 & 38704 & 38704 & 38704 & \mathbf{332731}
\\ \hline
\end{array}
\] }

\subsection{Recursion for $S_n(\tau)[k, \ell]$, where $\tau = 3154\-2$.}
\[ \begin{array}{|r|rrrrr|r|}
\hline
S_5(\tau)[k, \ell]
& 1 & 2 & 3 & 4 & 5 & \Sigma \\ \hline
k= 1 & 0 & 6 & 6 & 6 & 6 & 24 \\
k= 2 & 6 & 0 & 6 & 6 & 6 & 24 \\
k= 3 & 5 & 6 & 0 & 6 & 6 & 23 \\
k= 4 & 6 & 6 & 6 & 0 & 6 & 24 \\
k= 5 & 6 & 6 & 6 & 6 & 0 & 24 \\
\hline
\Sigma & 23 & 24 & 24 & 24 & 24 & \mathbf{119}
\\ \hline
\end{array}
\]
\[ \begin{array}{|r|rrrrrr|r|}
\hline
S_6(\tau)[k, \ell]
& 1 & 2 & 3 & 4 & 5 & 6 & \Sigma \\ \hline
k= 1 & 0 & 24 & 24 & 23 & 24 & 24 & 119 \\
k= 2 & 24 & 0 & 24 & 23 & 24 & 24 & 119 \\
k= 3 & 18 & 24 & 0 & 23 & 24 & 24 & 113 \\
k= 4 & 22 & 22 & 23 & 0 & 24 & 24 & 115 \\
k= 5 & 24 & 24 & 23 & 24 & 0 & 24 & 119 \\
k= 6 & 24 & 24 & 23 & 24 & 24 & 0 & 119 \\
\hline
\Sigma & 112 & 118 & 117 & 117 & 120 & 120 & \mathbf{704}
\\ \hline
\end{array}
\]
\[ \begin{array}{|r|rrrrrrr|r|}
\hline
S_7(\tau)[k, \ell]
& 1 & 2 & 3 & 4 & 5 & 6 & 7 & \Sigma \\ \hline
k= 1 & 0 & 119 & 119 & 113 & 115 & 119 & 119 & 704 \\
k= 2 & 119 & 0 & 119 & 113 & 115 & 119 & 119 & 704 \\
k= 3 & 83 & 119 & 0 & 113 & 115 & 119 & 119 & 668 \\
k= 4 & 101 & 101 & 113 & 0 & 115 & 119 & 119 & 668 \\
k= 5 & 113 & 113 & 107 & 115 & 0 & 119 & 119 & 686 \\
k= 6 & 119 & 119 & 113 & 115 & 119 & 0 & 119 & 704 \\
k= 7 & 119 & 119 & 113 & 115 & 119 & 119 & 0 & 704 \\
\hline
\Sigma & 654 & 690 & 684 & 684 & 698 & 714 & 714 & \mathbf{4838}
\\ \hline
\end{array}
\]
\[ \begin{array}{|r|rrrrrrrr|r|}
\hline
S_8(\tau)[k, \ell]
& 1 & 2 & 3 & 4 & 5 & 6 & 7 & 8 & \Sigma \\ \hline
k= 1 & 0 & 704 & 704 & 668 & 668 & 686 & 704 & 704 & 4838 \\
k= 2 & 704 & 0 & 704 & 668 & 668 & 686 & 704 & 704 & 4838 \\
k= 3 & 469 & 704 & 0 & 668 & 668 & 686 & 704 & 704 & 4603 \\
k= 4 & 563 & 563 & 668 & 0 & 668 & 686 & 704 & 704 & 4556 \\
k= 5 & 632 & 632 & 596 & 668 & 0 & 686 & 704 & 704 & 4622 \\
k= 6 & 680 & 680 & 644 & 644 & 686 & 0 & 704 & 704 & 4742 \\
k= 7 & 704 & 704 & 668 & 668 & 686 & 704 & 0 & 704 & 4838 \\
k= 8 & 704 & 704 & 668 & 668 & 686 & 704 & 704 & 0 & 4838 \\
\hline
\Sigma & 4456 & 4691 & 4652 & 4652 & 4730 & 4838 & 4928 & 4928 & \mathbf{37875}
\\ \hline
\end{array}
\]
{\scriptsize
\[ \begin{array}{|r|rrrrrrrrr|r|}
\hline
S_9(\tau)[k, \ell]
& 1 & 2 & 3 & 4 & 5 & 6 & 7 & 8 & 9 & \Sigma \\ \hline
k= 1 & 0 & 4838 & 4838 & 4603 & 4556 & 4622 & 4742 & 4838 & 4838 & 37875 \\
k= 2 & 4838 & 0 & 4838 & 4603 & 4556 & 4622 & 4742 & 4838 & 4838 & 37875 \\
k= 3 & 3119 & 4838 & 0 & 4603 & 4556 & 4622 & 4742 & 4838 & 4838 & 36156 \\
k= 4 & 3690 & 3690 & 4603 & 0 & 4556 & 4622 & 4742 & 4838 & 4838 & 35579 \\
k= 5 & 4136 & 4136 & 3901 & 4556 & 0 & 4622 & 4742 & 4838 & 4838 & 35769 \\
k= 6 & 4481 & 4481 & 4246 & 4199 & 4622 & 0 & 4742 & 4838 & 4838 & 36447 \\
k= 7 & 4719 & 4719 & 4484 & 4437 & 4503 & 4742 & 0 & 4838 & 4838 & 37280 \\
k= 8 & 4838 & 4838 & 4603 & 4556 & 4622 & 4742 & 4838 & 0 & 4838 & 37875 \\
k= 9 & 4838 & 4838 & 4603 & 4556 & 4622 & 4742 & 4838 & 4838 & 0 & 37875 \\
\hline
\Sigma & 34659 & 36378 & 36116 & 36113 & 36593 & 37336 & 38128 & 38704 & 38704 & \mathbf{332731}
\\ \hline
\end{array}
\]}

\acknowledgements
This research was funded by NSF grant 1358659 and NSA grant H98230-16-1-0026
as part of the 2016 Duluth Research Experience for Undergraduates (REU).

The authors warmly thank Andrew Baxter for many helpful conversations,
as well as for several pointers to the literature and known results.
The authors also thank Joe Gallian for supervising the research.
The authors would also like to thank both Baxter and Gallian for suggesting the problem.

Finally, the authors would like to thank Gallian and Levent Alpoge
as well as the referees, for helpful comments on drafts of the paper.

\bibliographystyle{plain}
\bibliography{wilf-refs.bib}

\end{document}